\documentclass[12pt, reqno]{amsart}

\newtheorem{theorem}{Theorem}[section]
\newtheorem{lemma}[theorem]{Lemma}
\newtheorem{proposition}[theorem]{Proposition}
   
\newtheorem{definition}[theorem]{Definition}
\newtheorem{example}[theorem]{Example}

\newtheorem{remark}[theorem]{Remark}

\newtheorem{question}[theorem]{Question}
\newtheorem{discussion}[theorem]{Discussion}

\numberwithin{equation}{section}
\newtheorem{thmx}{Theorem}
\newenvironment{discussionbox}[1][]{%
	\begin{discussion}[#1]\pushQED{\qed}}{\popQED \end{discussion}}

\usepackage{times}\usepackage{enumerate}
 \usepackage{mathrsfs}
\usepackage{tfrupee}
\usepackage{tikz}
\usetikzlibrary{chains,fit}
\usetikzlibrary{shapes,snakes}
\usetikzlibrary{graphs}
\usepackage{graphicx,adjustbox}
\usepackage{amsmath,amssymb}
\usepackage{amscd}
\usepackage{graphicx}
\usepackage[backend=bibtex,style=alphabetic, ]{biblatex}
\usepackage{color}
\usepackage[normalem]{ulem}
\usepackage[colorlinks=true,linkcolor=blue, citecolor=blue, urlcolor=blue, backref=true]{hyperref}
\usepackage{float}
\usepackage{caption}
\usepackage{subcaption}
\usepackage{cleveref}
\usepackage{geometry}
\usepackage[backend=bibtex]{biblatex}
\DeclareMathOperator{\height}{ht}
\DeclareMathOperator{\bight}{bight}
\DeclareMathOperator{\reg}{reg}
\DeclareMathOperator{\Projdim}{pd}
\DeclareMathOperator{\vnum}{\mathrm{v}}
\DeclareMathOperator{\Ass}{\mathrm{Ass}}
\DeclareMathOperator{\frakp}{\mathfrak{p}}

\DeclareMathOperator{\Ext}{Ext}
\DeclareMathOperator{\Hom}{Hom}

\begin{document}

\title[On the $\mathrm{v}$-number of Gorenstein ideals and Frobenius powers]{On the $\mathrm{v}$-number of Gorenstein ideals and Frobenius powers}
\author{Nirmal Kotal}
\author{
Kamalesh Saha
}

\date{}
\address{\small \rm Chennai Mathematical Institute, Siruseri, Chennai, Tamil Nadu 603103, India}
\email{nirmal@cmi.ac.in}

\address{\small \rm Chennai Mathematical Institute, Siruseri, Chennai, Tamil Nadu 603103, India}
\email{ksaha@cmi.ac.in}
\date{}



\allowdisplaybreaks

\begin{abstract}
In this paper, we show the equality of the (local) $\mathrm{v}$-number and Castelnuovo-Mumford regularity of certain classes of Gorenstein algebras, including the class of Gorenstein monomial algebras. Also, for the same classes of algebras with the assumption of level, we show that the (local) $\mathrm{v}$-number serves as an upper bound for the regularity. Moreover, we investigate the $\mathrm{v}$-number of Frobenius powers of graded ideals in prime characteristic setup. In this study, we demonstrate that the $\mathrm{v}$-numbers of Frobenius powers of graded ideals have an asymptotically linear behaviour. In the case of unmixed monomial ideals, we provide a method for computing the $\mathrm{v}$-number without prior knowledge of the associated primes.
\end{abstract}

\maketitle

\section{Introduction}

Let $I$ be an ideal in a Noetherian ring $S$. Then $I$ can be written as a finite irredundant intersection of primary ideals, which is known as the primary decomposition of $I$. Primary decomposition is an effective tool for investigating the algebra $S/I$. The associated primes of $I$, denoted by $\mathrm{Ass}(I)$, are the radicals of the primary ideals appearing in the primary decomposition of $I$. It is a well-known fact that associated primes of $I$ are precisely the prime ideals of the form $I:f$ for some $f\in S$. Let $S=K[x_{1}\,\ldots, x_{n}]=\bigoplus_{d=0}^{\infty}S_{d}$ denote the polynomial ring in $n$ variables over a field $K$ with standard grading. Then, for a graded ideal $I\subsetneq S$, the associated primes of $I$ are precisely the prime ideals of the form $I:f$ for some $f\in S_d$ and this defines the notion of $\mathrm{v}$-number.

\begin{definition}{\rm
Let $I$ be a proper graded ideal of $S$. Then the $\mathrm{v}$-\textit{number} of $I$, denoted by $\mathrm{v}(I)$, is defined as follows
		$$ \mathrm{v}(I):=
		\mathrm{min}\{d\geq 0 \mid \exists\, f\in S_{d}\,\,\text{and}\,\, \mathfrak{p}\in \mathrm{Ass}(I)\,\, \text{with}\,\, I:f=\mathfrak{p} \}.$$
		For each $\mathfrak{p}\in \mathrm{Ass}(I)$, we can locally define the $\mathrm{v}$-number as
		$$\mathrm{v}_{\mathfrak{p}}(I):=\mathrm{min}\{d\geq 0 \mid \exists\, f\in R_{d}\,\, \text{satisfying}\,\, I:f=\mathfrak{p} \}.$$
		Then $\mathrm{v}(I)=\mathrm{min}\{\mathrm{v}_{\mathfrak{p}}(I)\mid \mathfrak{p}\in\mathrm{Ass}(I)\}$. Note that $\mathrm{v}(I)=0$ if and only if $I$ is prime.
	}
\end{definition}

The motivation behind studying the $\vnum$-number has its foundation in coding theory. In 2020, Cooper et al. introduced the invariant $\mathrm{v}$-number \cite{cstpv20} to investigate the asymptotic behaviour of the minimum distance function of Reed-Muller-type codes. Let $\mathbb{X}$ be a  finite set of projective points, and $I(\mathbb{X})$ denote its vanishing ideal. Then it has been proved in \cite{cstpv20} that $\delta_{I(\mathbb{X})}(d)=1$ if and only if $\mathrm{v}(I(\mathbb{X}))\leq d$, where $\delta_{I(\mathbb{X})}$ denotes the minimum distance function of the projective Reed-Muller-type codes associated to $\mathbb{X}$. In \cite{js23}, the authors mentioned a geometrical point of view of local $\mathrm{v}$-numbers. Specifically, the local $\vnum$-number expands upon the concept of the degree of a point within a finite collection of projective points as presented in \cite{gkr93}.

Researchers investigate the $\mathrm{v}$-number from several perspectives, such as:
\begin{itemize}
	\item $\mathrm{v}$-number of monomial ideals (including edge ideals) in  \cite{civan23}, \cite{grv21}, \cite{vedge}, \cite{saha2023mathrmvnumber}, and \cite{ki_vmon}.
	\item $\mathrm{v}$-number of binomial edge ideals  in \cite{ski23}, and \cite{js23}.
	\item $\vnum$-number of powers of graded ideals in \cite{fs23} and \cite{ficarra2023simon}.
	\item $\vnum$-number as a lower bound of Castelnuovo-Mumford regularity (in short regularity) in \cite{ski23}, \cite{npv18}, \cite{cstpv20}, \cite{vedge}, \cite{saha2023mathrmvnumber}, and \cite{ki_vmon}.
\end{itemize}
Nevertheless, the question of whether the v-number serves as a lower bound for regularity continues to be a subject of current interest. Indeed, there was a conjecture in \cite[Conjecture 4.2]{npv18} whether $\mathrm{v}(I)\leq \reg S/I$ for any square-free monomial ideal. However, in \cite{vedge}, an edge ideal of a graph is provided as a counter-example. Therefore, the following question emerges as pertinent and intriguing to the researchers:
\begin{question}
	Let $I\subseteq S$ be a graded ideal. What conditions on $I$ will ensure the equality $\mathrm{v}(I)= \mathrm{reg} S/I$ or the inequality $\mathrm{v}(I)\geq \mathrm{reg} S/I$?
\end{question}
An affirmative answer to the question also entails providing a computational tool for determining the regularity or establishing an upper bound on the regularity. There is a limited number of studies in the existing literature in this direction. For example, it has been shown that if $I$ is a complete intersection monomial ideal \cite[Proposition 3.10]{ki_vmon}, or if $S/I$ is a certain level algebra of dimension at most one \cite{cstpv20}, then the equality $\vnum(I)=\reg S/I$ holds. This paper aims to investigate the question with potentially broader applicability. Our main results in this regard are as follows:
\begin{thmx}{(Theorem \ref{thm:gor_vnum_equal_reg},\ref{thm:level_vnum})}
	Let $I\subset S$ be a graded ideal and $\mathfrak{p}\in \Ass(I)$ be an associated prime of $I$ generated by linear forms. Then, the following hold.
	\begin{enumerate}
		\item If $S/I$ is Gorenstein, then $\vnum_{\mathfrak{p}}(I)=\reg S/I$.
		
		\item If $S/I$ is level, then $\vnum_{\mathfrak{p}}(I)\geq\reg S/I$.
	\end{enumerate}
	More specifically, if all the associated primes of $I$ are generated by linear forms (for example, if $I$ is a monomial ideal), then we can replace the local $\vnum$-number $\vnum_{\frakp}(I)$ by $\vnum(I)$ in above statements.
\end{thmx}
\noindent We present examples to support the theory stated above. Specifically, we give a level monomial ideal $I$ for which $\vnum(I)>\reg S/I$ (see Example \ref{exmp:level_vi_geq}), as well as a Cohen-Macaulay monomial ideal $I$ for which $\vnum(I)>\reg S/I$ (see Example \ref{exmp:cm:vi}).

\medskip

In the second part of this article, we study the asymptotic behaviour of the $\vnum$-number of Frobenius power of graded ideals. Over the course of time, researchers have conducted extensive investigations on several algebraic invariants, including depth, regularity, projective dimension, Betti numbers, etc., associated with the usual power and the Frobenius power of a graded ideal. Given the novelty of the $\vnum$-number notion, there exists just a single work \cite{fs23} that investigates the asymptotic properties of the $\vnum$-number pertaining to powers of graded ideals. This paper aims to address the existing research gap by conducting an investigation into the $\vnum$-number of Frobenius powers. The following are some significant results of this section.
\begin{thmx}{(Theorem \ref{thm:assymptotic_of_vi}, Proposition \ref{propviqlower})}
	Let $S$ be a polynomial ring over a field of prime characteristic $p$, and in this context, $q$ is always a power of $p$. Let $I\subset S$ be a graded ideal and the $q$-th Frobenius power of $I$, defined as $I^{[q]}:=\left( a^q : a\in I\right)$. Then the following results hold
	\begin{enumerate}
		\item $\vnum(I^{[q]})\geq q\vnum(I)$ for all $q\geq 1$ and hence $\left\{\frac{\vnum(I^{[q]})}{q}\right\}$ is a non-decreasing sequence in $q$, where $ q=p^e, e\in \mathbb{N}$.
		\item $\displaystyle{\lim_{q \to \infty}} \frac{\vnum(I^{[q]})}{q}$ exists. 
		\item If $I$ is an unmixed monomial ideal, then $\vnum(I^{[q]})=q\vnum(I)+(q-1)\height(I)$ for all $q\geq 1$.
	\end{enumerate}
\end{thmx}
To prove the above theorem, we introduce a new invariant $\alpha_{q}(I)$ as follows: 
\begin{align*}
	\alpha_q (I)=\min\left\{d \mid \left[\frac{I^{[q]}:I}{I^{[q]}} \right]_d\neq 0 \right\}.
\end{align*} 
$\alpha_{q}(I)$ helps us to obtain an upper bound for $\vnum(I^{[q]})$ (see Proposition \ref{propviqupper}). In  Theorem \ref{thmalphaq}, we show that $\lim_{q\to \infty}\frac{\alpha_{q}(I)}{q}$ exists. Also, we show that if $I$ is radical, then $\lim_{q\to \infty}\frac{\alpha_{q}(I)}{q}=\lim_{q\to \infty}\frac{\vnum(I^{[q]})}{q}$. Due to Theorem \ref{thmalphaq}(\ref{thm:alph:ceil_vi}) and polarization technique, we prove for an unmixed monomial ideal $I$ that $\vnum(I)=\lceil \frac{\alpha_{q}(I^{\mathcal{P}})}{q}\rceil-\height (I)$ when $q>\dim S^{\mathcal{P}}$, where $I^{\mathcal{P}}$ is the polarization of $I$ and $S^{\mathcal{P}}$ is the corresponding polynomial ring of $I^{\mathcal{P}}$ (see Remark \ref{rmk:vnum_pol}). Therefore, for an unmixed monomial ideal $I$, by investigating $\alpha_{q}(I)$, we can compute $\vnum(I)$ without knowing the primary decomposition. 

The paper is structured as follows. Section \ref{preli} provides an overview of the necessary prerequisites pertaining to our study. In Section \ref{sec3}, we establish the relation between the $\vnum$-number and the regularity of a wide range of Gorenstein and level ideals. Section \ref{sec4} delves into an examination of the $\vnum$-number of Frobenius power of graded ideals. Finally, in Section \ref{sec:question}, we pose some questions for potential future investigation.

\section{Preliminaries}\label{preli}
A \emph{monomial} in the polynomial ring $S$ is defined as a polynomial of the form $x_{1}^{a_{1}}\cdots x_{n}^{a_{n}}$, where each $a_{i}$ is a non-negative integer. A \emph{monomial ideal} $I\subset S$ is defined as an ideal that is generated by a set of monomials in the ring $S$. The set of minimal monomial generators of $I$ is unique, and if it consists of square-free monomials, then we say $I$ is a \emph{square-free} monomial ideal.
Let $G=(V(G),E(G))$ be a simple graph with $V(G)=\{x_1,\ldots,x_n\}$. Then the \emph{edge ideal} of $G$, denoted by $I(G)$, is a square-free monomial ideal in $S$ defined as $I(G):=(\{x_{i}x_{j}\mid \{x_{i},x_{j}\}\in E(G)\})$.

The \emph{height} (respectively, \emph{big height}) of an ideal $I\subset S$, denoted by $\height(I)$ (respectively, $\bight(I)$), is the minimum (respectively, maximum) height among all the associated prime of $I$. The ideal $I$ is said to be \emph{unmixed} if $\height(I)=\bight(I)$. An ideal $I\subset S$ is called a \emph{complete intersection} if $I$ is generated by a regular sequence. If $I$ is a complete intersection, then the height of $I$ is the cardinality of a minimal generating set of $I$. The following observation of the ideal generated by linear forms is widely known. For the reader's benefit, we provide a short proof here.
\begin{remark}
	\label{rmk:ci}
	{\rm Let $I\subset S$ be a graded ideal minimally generated by linear forms. Then $I$ is a complete-intersection prime ideal. Moreover $\reg S/I =0$.}
\end{remark}
\begin{proof}
	Let $l_1,\ldots,l_m$ are the linear forms that minimally generate the ideal $I$. Without loss of generality we may assume $l_i=x_i+c_i$ for all $1\leq i\leq m$, where $c_1,\ldots,c_m$ are linear polynomials involving none of the variables $x_1,\ldots,x_m$ \cite[Exercise 10]{CLO2015_ideals}. Now consider  the automorphism $\phi:S\to S$ defined as 
	$$\phi(x_i)=\begin{cases}
		x_i-c_i&\text{if } 1\leq i\leq m;\\
		x_i&\text{else.}
	\end{cases}$$ Rest of the proof follows from the fact that $\phi(I)=(x_1,\ldots,x_m)$.
\end{proof}

Let $I\subset S$ be a graded ideal and denote $R:=S/I$. Let $M=\bigoplus_{i\in \mathbb{Z}} M_i$ be a finitely generated graded $S$-module. Let $\alpha(M)$ denote the minimum degree of a non-zero element in $M$, that is $\alpha(M)=\min\{i\mid M_i\neq 0 \}$. For an integer $j$, the $j$-th \emph{shift} module $M(j)$ is defined by the grading $M(j)_i=M_{i+j}$. The \emph{Hilbert series} of $M$ defined by $H(M,t)=\sum_{i}\dim_KM_it^i$ is a power series in $\mathbb{Z}[t,t^{-1}]$. If $M$ is positively graded, then the Hilbert series of $M$ can be written as $H(M,t)=\frac{h(t)}{(1-t)^{\dim M}}$ for some polynomial $h(t)\in \mathbb{Z}[t]$.

Let $I\subset S$ be a graded ideal and let $R:=S/I$ admit the following minimal free resolution:
		\begin{align*}
			%\mathbb{F}
			\mathbf{F}_{\bullet}: \quad	0\rightarrow F_c\rightarrow \cdots \rightarrow F_1\rightarrow F_0\rightarrow 0.
		\end{align*}
So $F_0=S$ and since $R$ is graded, for each $1\leq i\leq c$, $F_i$ is of the form: $F_i=\bigoplus_{j} S(-j)^{\beta_{i,j}}$ for some integers $j,\beta_{i,j}$. The number $\beta_{i,j}$ is called the $(i,j)$-th graded \emph{Betti number} of $R$.
	The \emph{Castelnuovo-Mumford regularity} of $R$ (in short, \emph{regularity} of $R$) is denoted by $\reg R$ and defined as follows
	\begin{align*}
		 \reg R:=\max\left\{ j-i\mid \beta_{i,j}\neq 0\right\}.
	\end{align*}
		The \textit{projective dimension} of $R$, denoted by $\Projdim R$, is defined as follows
		\begin{align*}
			\Projdim R:=\max\left\{i\mid \beta_{i,j}\neq 0 \text{ for some } j\right\} =c.
		\end{align*}
In the following Discussion \ref{disc1}, we assume $R$ to be Cohen-Macaulay. 

\begin{discussionbox}\label{disc1}{\rm
Since $R$ is Cohen-Macaulay, by the Auslander-Buchsbaum theorem, we get $\Projdim R= \height (I)$. The canonical module of $R$, denoted as $\omega_R$, can be defined as $\omega_R=\Ext^c_S(R,S)$ \cite[Theorem 3.3.7]{bruns_herzog_1998}. More precisely, let $\mathbf{G}_{\bullet}=\Hom_S(F_{\bullet},S)$ be the following dual complex
	\begin{align*}
		\mathbf{G}_{\bullet}: \quad	0\rightarrow G_c\rightarrow \cdots \rightarrow G_1\rightarrow G_0\rightarrow 0,
	\end{align*}
where $G_i=\Hom_S(F_{c-i},S)$ for $0\leq i \leq c$. Then $\mathbf{G}_{\bullet}$ is the minimal free resolution of $\omega_R$ \cite[Corollary 3.3.9]{bruns_herzog_1998}. Define the \emph{$a$-invariant} of $R$ as
	\begin{align*}
		a(R):=-\min \left\{i \mid \left[\omega_R \right]_i\neq 0 \right\}.
	\end{align*}
	The ring $R$ is said to be \emph{Gorenstein} (sometimes, we say $I$ is Gorenstein) if its canonical module is cyclic, i.e. generated by a single element. This is the same as saying the rank of $G_0$ is one or, equivalently, $F_c=S(-(c+\reg R))$. The ring $R$ is said to be a \emph{level} ring (sometimes, we call $I$ is level) if every element in a minimal set of generators of the canonical module possesses the same degree. This is equivalent to the fact that $G_0$ has a basis consisting of same degree elements or equivalently, $F_c$ is of the form $F_c=S(-(c+\reg R))^{\beta_{c,c+\reg R}}$.

	The first syzygy of the canonical module denoted as $\mathrm{Syz}^1_S(\omega_R)$, is the kernel of the map $G_0\to \omega_R$. 
	
	Let $I\subsetneq J\subset S$ be two ideals such that $I$ is Gorenstein. Then $I:(I:J)=J$.
	}
\end{discussionbox}

The following observation of $\vnum$-number is utilized multiple times in the proofs, and hence, it is explicitly stated here for clarity.
\begin{remark}
	\label{remark:v_Iless_deg_f}
	{\rm Let $I\subset S$ be a graded non-prime ideal. Let $f\in S\setminus I$ be a homogeneous element such that $f\frakp \subseteq I$ for some associated prime $\frakp$ of $I$. If $\frakp$ is not contained in any associated prime of $I$, then $I:f=\frakp$, and hence $\vnum(I)\leq  \vnum_{\frakp}(I)\leq \deg f$.}
\end{remark}
\begin{proof}
	Since $f\notin I$, $I:f$ is a proper ideal of $S$. Let $\mathfrak{p'}\in \Ass(I:f)$. Then $\frakp \subseteq I:f \subseteq \mathfrak{p'}$. But, $\Ass(I:f)\subset \Ass(I)$ and hence, $\frakp=\mathfrak{p'}$. Therefore, $I:f=\frakp$.
\end{proof}

\section{The $\vnum$-number of Gorenstein and level ideals}\label{sec3}

In this section, we establish a relation between the $\vnum$-number and Castelnuovo-Mumford regularity of certain classes (including the class of monomial ideals) of Gorenstein and level algebras.\par 

The following result from Peskine and Szpiro serves as the foundation for the notion of algebraic linkage theory \cite[Proposition 2.6]{PeskineSzpiroLiaison1974}. To accomplish our goals, a slight modification of the result is necessary, as described in \cite[Section 1]{Kustin1984DeformationAL}.
\begin{proposition}
	\label{prop:gor_linkage}
	Let $I\subsetneq J$ be two homogeneous ideals of the same projective dimension $c$ such that the quotient rings $S/I$ and $S/J$ are Gorenstein. Let $\mathbf{F}_{\bullet}$ and $\mathbf{G}_{\bullet}$ be the minimal homogeneous free resolutions of $S/I$ and
	$S/J$, respectively, and let $\pi_\bullet:\mathbf{F}_{\bullet}\to \mathbf{G}_{\bullet}$ be a homogeneous map of resolutions which extends the natural surjective map $\pi:S/I\to S/J$. Since $S/I,S/J$ are Gorenstein, $F_c=S(-(c+r_I))$ and $G_c=S(-(c+r_J))$, where $r_I$ and $r_J$ are the regularity of $S/I$ and $S/J$ respectively. So the map $\pi_c : F_c\to  G_c$ is multiplication by a homogeneous element $f$ of $S$ and $\deg f=r_J-r_I$.
	Then $I:J=(I,f)$ and $I:f=J$.
\end{proposition}
We now employ the notion of linkage in order to establish the main result of this section.
\begin{theorem}
	\label{thm:gor_vnum_equal_reg}
	Let $I\subset S$ be a graded ideal such that $S/I$ is a Gorenstein algebra. If $\frakp\in \Ass(S/I)$ is generated by linear forms, then $\vnum_{\frakp}(I)=\reg S/I$. In particular, if all the associated primes of $I$ are generated by linear forms, then $\vnum(I)=\reg S/I$.
\end{theorem}
\begin{proof}
	If $I$ is itself a prime ideal generated by linear forms, then $I$ is a complete intersection and hence $\reg S/I =0$. Of course, if $I$ is a prime ideal, then $\vnum(I)=0$.
	
	So, assume $I$ is not a prime ideal. If $\frakp \in \Ass (S/I)$ is generated by linear forms, then $\frakp$ is a complete intersection and $\reg S/\frakp=0$ (Remark \ref{rmk:ci}). Note that they have the same projective dimension, which is equal to $\height(I)$. By Proposition~\ref{prop:gor_linkage}, there exist a homogeneous element $f$ of degree $\reg S/I$ such that $I:\frakp = (I,f)$, and $I:f=\frakp$. Hence $\vnum_{\frakp}(I)=\deg f = \reg S/I$.
	
	Since $\vnum(I)=\min\{\vnum_{\frakp}(I): \frakp \in \Ass (S/I)  \}$, the second assertion follows immediately.
\end{proof}
\begin{remark}{\rm
	If $I$ is a monomial graded ideal such that $S/I$ is Gorenstein, then $\vnum(I)=\reg S/I$. Since the primary decomposition of the monomial ideal is independent of the characteristic of the field, so is the $\vnum$-number. Thus, if $S/I$ is Gorenstein, then regularity is also independent of the characteristic of the field.}
\end{remark}
If $I\subset S$ be a graded ideal such that $S/I$ is a Gorenstein algebra and $\frakp\in \Ass(S/I)$ be such that $\frakp$ is generated by linear forms, then $\vnum_{\frakp}(I)=\reg S/I$ by Theorem \ref{thm:gor_vnum_equal_reg}. However, $\vnum(I)$ might be strictly less than $\reg S/I$. For instance, we consider the following example of Gorenstein binomial edge ideals.
\begin{example}
	{\rm 
	Let $G$ be a simple graph with $V(G)=\{1,\ldots,n\}$. Then the binomial edge ideal of $G$, denoted by $J_{G}$, is defined as follows:
$$J_{G}:=(\{x_iy_j-x_jy_i\mid \{i,j\}\in E(G)\text{  with  } i<j\})$$
in the polynomial ring $R=K[x_1,\ldots,x_n,y_1,\ldots,y_n]$. It has been proved in \cite[Theorem A]{rene21} that the only Gorenstein binomial edge ideals are the binomial edge ideals of path graphs. Now, consider the path graph $P_{2k}$ of even length $2k$. Then, $V(P_{2k})=\{1,\ldots,2k+1\}$. From the primary decomposition of binomial edge ideal given in \cite{HHHKR10binom}, it follows that $P_{2k}$ has an associated prime ideal generated in linear form and that is $\mathfrak{p}=(x_2,y_2,x_4,y_4,\ldots, x_{2k},y_{2k})$. Also, by \cite[Corollary 2.7]{EZ15regularity_path}, we have $\reg R/J_{P_{2k}}=2k$. Thus, by Theorem \ref{thm:gor_vnum_equal_reg}, we get $\vnum_{\mathfrak{p}}(J_{P_{2k}})=2k$. While, we can observe using Macaulay2 \cite{M2}, $\vnum(J_{P_{6}})=4<\reg R/J_{P_{6}}$.}
\end{example}

The Nagata idealization (also known as trivial extension) provides a valuable method for constructing Gorenstein rings from level rings. Consequently, we employ this technique to investigate the $\vnum$-number of level rings. Let $R=S/I$ be a standard graded level algebra and $\omega_R$ be the canonical module. Denote the $a$-invariant of $R$ as $a$. Consider the following ring obtained from Nagata idealization with its canonical module:
\begin{align*}
	\widetilde{R}:=R\ltimes \omega_R(-a-1).
\end{align*}
The addition and multiplication structure is given by $(r_1,z_1)+(r_2,z_2)=(r_1+r_2,z_1+z_2)$ and $(r_1,z_1)\cdot (r_2,z_2)=(r_1 r_2,r_1z_2+r_2z_1)$, for all $r_i\in R,z_i\in \omega_R(-a-1),i=1,2$. We state some observations of $\widetilde{R}$ here. Readers are encouraged to review Section 3 of \cite{MSS21quadratic_gorenstein} in order to enhance their comprehension.
\begin{proposition}
	\label{prop:level_to_gor}
	Assume the above setup. The following statements hold:
	\begin{enumerate}
		\item\label{level:gor} $\widetilde{R}$ is a standard graded Gorenstein algebra.
		\item\label{level:description} If $\omega_R$ is minimally generated by $m$ elements, then
		\begin{align*}
			\widetilde{R}\simeq \frac{S[y_1,\ldots,y_m]}{I+\mathcal{L}+(y_1,\ldots,y_m)^2},
		\end{align*}
		where $\mathcal{L}=(\sum f_iy_i : f_1,\ldots f_m \in \mathrm{Syz}^1_S(\omega_R))$.
		\item \label{level:hilbert} If the Hilbert series of $R$ is $H(R,t)=\frac{\sum_{i=0}^{r}a_it^i}{(1-t)^d}$, with $a_r\neq 0$, then the Hilbert series of $\widetilde{R}$ is $H(\widetilde{R},t)=\frac{a_0+\sum_{i=1}^{r}(a_i+a_{r-i})t^i+a_rt^{r+1}}{(1-t)^d}$, where $d=\dim R$.
		\item\label{level:regularity} $\reg \widetilde{R}=\reg R+1$.
	\end{enumerate}
\end{proposition}
\begin{proof}
	$(\ref{level:gor}):$ Proved in \cite[Theorem 7]{reiten72gorenstein}.
	
	$(\ref{level:description}):$ Shown as a part of the \cite[Lemma 3.3]{MSS21quadratic_gorenstein}.
	
	$(\ref{level:hilbert})$ First notice that $\dim \widetilde{R}=\dim R =d$ (say) \cite[Exercise 3.3.22]{bruns_herzog_1998}.
	Since $\dim_k\widetilde{R}_i=\dim_k R_i+\dim_k \omega_R(-a-1)_i$, so the Hilbert series $H(\widetilde{R},t)=H(R,t)+t^{a+1}H(\omega_R,t)$. Also $H(\omega_R,t)=(-1)^dH(R,t^{-1})$ \cite[Corollary 4.4.6]{bruns_herzog_1998}. Thus $H(\widetilde{R},t)=\frac{\sum_{i=0}^{r}a_it^i}{(1-t)^d}+t^{a+1+d-r}\frac{\sum_{i=0}^{r}a_{r-i}t^i}{(1-t)^d}$. The result follows immediately from the fact $a=r-d$.
	
	$(\ref{level:regularity})$ For a Cohen-Macaulay ring, the regularity is the degree of the polynomial in the numerator of the Hilbert series.  Hence $\reg(\widetilde{R})=\reg(R)+1$. 
\end{proof}

\begin{theorem}
	\label{thm:level_vnum}
	Let $I\subset S$ be a graded ideal such that $S/I$ is a level ring. If $\frakp\in \Ass(S/I)$ is generated by linear forms, then $\vnum_{\frakp}(I)\geq \reg S/I$. In particular, if all the associated primes of $I$ are generated by linear forms, then $\vnum(I)\geq \reg S/I$.
\end{theorem}
\begin{proof}
Denote $S/I$ as $R$ and the canonical module as $\omega_R$. Following Proposition \ref{prop:level_to_gor}, $\widetilde{R}= \frac{S[y_1,\ldots,y_m]}{I+\mathcal{L}+(y_1,\ldots,y_m)^2}$ is a standard graded Gorenstein ring, where $\omega_R$ is minimally generated by $m$ elements and $\mathcal{L}=(\sum f_iy_i : f_1,\ldots ,f_m \in \operatorname{Syz}^1_S(\omega_R))$. Denote the ideal $I+\mathcal{L}+(y_1,\ldots,y_m)^2$ by $\widetilde{I}$. Let $\widetilde{\frakp}= \frakp+(y_1,\ldots,y_m)$. Note that $\widetilde{\frakp}$ is an associated prime of $\widetilde{I}$ generated by linear form. Then $\vnum_{\widetilde{\frakp}}(\widetilde{I})=\reg \widetilde{R}$ by Theorem \ref{thm:gor_vnum_equal_reg}. But $\reg \widetilde{R}=\reg R+1$.

Now let $f\in S$ such that $I:f=\frakp$ and $\deg f =\vnum_{\frakp}(I)$. Then $fy_i\widetilde{\frakp} \subseteq \widetilde{I}$ for all $1\leq i\leq m$. If $fy_i\notin \widetilde{I}$ for some $i$, then 
$\vnum_{\widetilde{\frakp}}(\widetilde{I})\leq \deg fy_i=\vnum_{\frakp}(I)+1$ (by Remark \ref{remark:v_Iless_deg_f}). Else, if $fy_i\in \widetilde{I}$ for all $1\leq i\leq m$, then $f\widetilde{\frakp}\subseteq \widetilde{I}$. However, it is important to note that $f$ is not an element of $I$, and hence it does not belong to $\widetilde{I}$ either. Hence 
$\vnum_{\widetilde{\frakp}}(\widetilde{I})\leq \deg f=\vnum_{\frakp}(I)$. In both cases $\vnum_{\widetilde{\frakp}}(\widetilde{I})\leq \vnum_{\frakp}(I)+1$. Therefore, we get $\vnum_{\frakp}(I) \geq \reg R $.
\end{proof}

Following results are due to \cite[Corollary 4.4, Theorem 4.10]{cstpv20}. As the results are relevant to Theorem \ref{thm:level_vnum}, we state here for the benefit of the readers.
\begin{theorem}
	Let $I\subset S$ be a graded ideal.
	\begin{enumerate}
		\item Assume $S/I$ is artinian. Then $\vnum(I)\leq \reg S/I$. Furthermore, the equality holds if and only if $S/I$ is a level algebra.
		\item Assume $\dim S/I=1$, $I$ is unmixed and all the associated primes of $I$ are minimally generated by linear forms. Then $\vnum(I)\leq \reg S/I$. Furthermore, the equality holds if $S/I$ is a level algebra. 
	\end{enumerate} 
\end{theorem}

The following examples show that the  $\vnum$-number can be strictly greater than the regularity for a level algebra. Indeed, the difference between the $\vnum$-number and the regularity of a level algebra can be arbitrarily large.
\begin{example}
	\label{exmp:level_vi_geq}
	{\rm Take the graph $G$ from \cite[Example 5.4]{vedge}. Let $I$ be the edge ideal of $G$. Then
	$S=K[x_1,\ldots,x_{11}]$ and
	\begin{align*}
		I=&(x_1 x_3 , x_1 x_4 , x_1 x_7 , x_1 x_{10}, x_1 x_{11} , x_2 x_4 , x_2 x_5 , x_2 x_8 , x_2 x_{10} ,
		x_2 x_{11} , x_3 x_5 , x_3 x_6 , x_3 x_8 ,\\& x_3 x_{11} , x_4 x_6 , x_4 x_9 ,
		x_4 x_{11} , x_5 x_7 ,
		x_5 x_9 , x_5 x_{11} , x_6 x_8 , x_6 x_9 , x_7 x_9 , x_7 x_{10} , x_8 x_{10}).
	\end{align*}
	The computation conducted using Macaulay2 \cite{M2} demonstrates:
	\begin{enumerate}
		\item When $K=\mathbb{Q}$ then $S/I$ is Cohen-Macaulay, $\vnum(I)=3$ and $\reg S/I=2$. The Betti numbers  $$\beta_{c,j} =\begin{cases}
			11 & \text{ if } j=c+\reg S/I\\
			0	& \text{ else,}
		\end{cases}$$
		where $c$ is the projective dimension (here $c=8$). That is, the last free module $F_c$ in the free resolution of $S/I$, has the same degree. Consequently, $S/I$ is level.
		
		\item When $K=\mathbb{F}_2$ (finite field of cardinality two), then $\vnum(I)=3$ and $\reg S/I=3$, but $S/I$ is not even Cohen-Macaulay.
	\end{enumerate}}
\end{example}
\begin{example}{\rm
		Consider the graph $H=G_1\sqcup \ldots \sqcup G_k$ with each $G_i$ isomorphic to the graph $G$ mentioned in Example \ref{exmp:level_vi_geq} for all $1\leq i \leq k$. Then by \cite[Proposition 3.9]{ki_vmon}, $\vnum(I(H))=3k$. Again, we have $\reg \frac{\mathbb{Q}[V(H)]}{I(H)}=2k$. Since $\frac{\mathbb{Q}[V(G)]}{I(G)}$ is level, so is $\frac{\mathbb{Q}[V(H)]}{I(H)}$. Thus, the difference between the $\vnum$-number and regularity can be arbitrarily large for level algebras.
}
\end{example}

\begin{example}\label{exmp:cm:vi}
	{\rm Let $G$ be a simple graph with $E(G)=\{\{x_1,x_i\}\mid 1<i\leq n\}$ and $n\geq 3$. Let $W_G$ be the whisker graph on $G$, i.e., $V(W_G)=V(G)\cup\{y_1,\ldots,y_n\}$ and $E(W_G)=E(G)\cup\{\{x_i,y_i\}\mid 1\leq i\leq n\}$. Then we observe that $I(W_G):x_1=(x_2,\ldots,x_n,y_1)$. Hence $\vnum(I(W_G))=1$. The ring $\frac{K[V(W_G)]}{I(W_G)}$ is not level and $\reg \frac{K[V(W_G)]}{I(W_G)}=n-1$ (see \cite[Proposition 2.10]{nguyen2023weak}). Also, it is well-known that whisker graphs are Cohen-Macaulay. Thus, the regularity can be arbitrarily larger than the $\vnum$-number for Cohen-Macaulay edge ideals.}
\end{example}

\section{The $\vnum$-number of Frobenius powers}\label{sec4}

In this section, $S=K[x_1,\ldots,x_n]$ be a standard graded polynomial ring over a field $K$ of prime characteristic $p$, and in this context, $q$ is always a power of $p$. That is $q=p^e$ for some non-negative integer $e$. Also, assume that $I\subset S$ be a graded ideal. Define the $q$-th Frobenius power of $I$ as $I^{[q]}:=\left( a^q : a\in I\right)$.

The primary objective of this section is to comprehend the asymptotic behaviour of $\vnum(I^{[q]})$. In order to achieve our goal, we introduce an invariant as follows. For each $q> 1$, we define 
\begin{align*}
	\alpha_q (I):=\min\left\{d \mid \left[\frac{I^{[q]}:I}{I^{[q]}} \right]_d\neq 0 \right\}.
\end{align*}
Observe that $\alpha_{q}(I)$ is same as $\alpha((I^{[q]}:I)/I^{[q]})$. The subsequent portion of this section will delve into the asymptotic behaviour of $\alpha_{q}(I)$ and its connection with $\vnum(I^{[q]})$.

Before we start, we state some important results related to Frobenius powers.
\begin{lemma}
	Assume the above notation. For any ideal $I,J\subset S$ and for all $q\geq 1$, we have
	\begin{enumerate}
		\item $(I\cap J)^{[q]}=I^{[q]}\cap J^{[q]}$ and $(I:J)^{[q]}=I^{[q]}:J^{[q]}$.
		\item \cite[Lemma 2.2]{hh02comparison_of_symbolic} $\Ass(I^{[q]})=\Ass (I)$.
	\end{enumerate}
\end{lemma}
One noteworthy observation, proved below, is that the $\alpha_{q}(I)$ is bounded above by a linear function.
\begin{lemma}
	\label{lemma:alpha_q_bounded}
	Let $I\subset S$ be a graded ideal minimally generated by homogeneous elements $g_1,g_2,\ldots,g_m$. Then for all $q>1$, $\alpha_{q}(I)\leq (q-1)\sum_{i=1}^{m}\deg g_i$.
\end{lemma}
\begin{proof}
	Let $g=\prod_{i=1}^{m}g_i^{q-1}$. Note that we need to show that $\alpha_{q}(I)$ is bounded above by  $\deg g$. As  $gg_i \in g_i^qS$ for all $1\leq i \leq m$, so $gI\subseteq I^{[q]}$. If $g\not\in I^{[q]}$, then we are done. So assume $g\in I^{[q]}$.
	
	For each $1\leq l\leq m(q-1)$, consider the set of homogeneous elements $$\chi(l):=\{ \prod_{i=1}^{m}g_i^{a_i} : 0\leq  a_i \leq q-1 \text{ and } \sum_{i=1}^{m}a_i=l \}.$$
	For a fixed $l\geq 2$, we claim that,
\begin{align}
	\label{claim:induction_on_l}
	\text{ if }
		\chi(l) \subseteq I^{[q]} \text{ then } \chi(l-1)\subseteq I^{[q]}:I. 
\end{align}
Assume the claim. Since $g\in I^{[q]}$, this mean $\chi(m(q-1))\subseteq I^{[q]}$. Therefore $\chi (m(q-1)-1)\subseteq I^{[q]}:I$. If $\chi (m(q-1)-1)\not\subseteq I^{[q]}$, then there exist an $h\in \chi (m(q-1)-1)$ such that $h\in (I^{[q]}:I)\setminus I^{[q]}$ and hence $\alpha_q(I) \leq \deg h \leq \deg g$. Else $\chi(m(q-1)-1)\subseteq I^{[q]}$. We repeat the argument until we get an $l'\geq 2$, such that $\chi (l'-1)\subseteq I^{[q]}:I$ and $\chi (l'-1)\not\subseteq I^{[q]}$. The inductive process must stop and such an $l'$ always exists because $\chi(1) \not\subseteq I^{[q]}$. Hence there exists an $h\in \chi(l'-1)$ such that $h\in (I^{[q]}:I)\setminus I^{[q]}$. Therefore $\alpha_{q}(I) \leq \deg h \leq  \deg g=(q-1)\sum_{i=1}^{m}\deg g_i$.

	It remains to prove the claim~\ref{claim:induction_on_l}. Assume that for some $l\geq 2$, $\chi(l)\subseteq I^{[q]}$. Let $h\in \chi(l-1)$. If $g_i^{q-1}$ is a factor of $h$, then $hg_i\in g_i^{q}S\subseteq I^{[q]}$. Else $hg_i\in \chi(l)\subseteq I^{[q]}$. Therefore $hg_i\in I^{[q]}$ for all $1\leq i\leq m$ that is $h \in I^{[q]}:I$. Therefore $\chi(l-1)\subseteq I^{[q]}$. Hence, the claim follows.
\end{proof}

\begin{lemma}
	\label{lemma:alphaq_equal_vnum}
	Let $\frakp\subset S$ be a graded prime ideal. Then for all $q>1$, $\vnum(\frakp^{[q]})=\alpha_{q}(\frakp)$.
\end{lemma}
\begin{proof}
	Since $\frakp$ is the only associated prime of $\frakp^{[q]}$, $\vnum(\frakp^{[q]})=\alpha \left( \frac{\frakp^{[q]}:\frakp}{\frakp^{[q]}} \right)$ \cite[Proposition 4.2]{cstpv20}, which is same as $\alpha_{q}(\frakp)$.	
\end{proof}
\begin{lemma}
	\label{lem:vi_ci}
	Let $I\subset S$ be a graded ideal minimally generated by linear forms. Then for all $q>1$, $$\vnum(I^{[q]})=\alpha_q(I) =(q-1)\height(I).$$
\end{lemma}
\begin{proof}
	Keep in mind that $I$ is both a complete intersection and a prime ideal. Assume that $I$ is minimally generated by linear forms $l_1,\ldots,l_m$. Then $I^{[q]}=(l_1^q,\ldots,l_m^q)$. The sequences $l_1,\ldots,l_m$ and $l_1^q,\ldots,l_m^q$ both are regular sequences. Using \cite[Corollary 2.3.10]{bruns_herzog_1998}, $I^{[q]}:I=I^{[q]}+(l_1\cdots l_k)^{q-1}S$ and consequently $I^{[q]}:(l_1\cdots l_m)^{q-1}=I$. Hence $\alpha_{q}(I)=\deg (l_1\cdots l_m)^{q-1}=(q-1)\height(I)$. Now, use Lemma \ref{lemma:alphaq_equal_vnum} to get the desired result.
\end{proof}
Now, we are ready to show that the $\vnum$-number of Frobenius power is bounded by a linear function.
\begin{proposition}\label{propviqupper}
	Let $I\subset S$ be a graded ideal and $\frakp \in \Ass I$. Then for all $q>1$,
	$$\vnum(I^{[q]})\leq q\vnum_{\frakp}(I)+(q-1) \sum_{i=1}^{m}\deg g_i,$$
	where $\frakp$ is minimally generated by $g_1,\ldots,g_m$.
	
	In particular, if all the associated primes of $I$ are generated by linear forms, then
	$$\vnum(I^{[q]})\leq q\vnum(I)+(q-1)\bight(I).$$
\end{proposition}

\begin{proof}
	Let $f\in S$ be a homogeneous polynomial such that $(I:f)=\frakp$ and $\deg f=\vnum_{\frakp}(I)$. Also let $h\in S$ be a homogeneous polynomial such that $(\frakp^{[q]}:h)=\mathfrak{p}$ and $\deg h=\vnum(\frakp^{[q]})$.
	Now
	\begin{align*}
		(I^{[q]}:f^qh) &=(I^{[q]}:f^{q}):h\\
		&=(I:f)^{[q]}:h\\
		&=\mathfrak{p}^{[q]}:h\\
		& =\mathfrak{p}
	\end{align*}
	Therefore, $\vnum(I^{[q]})\leq  q\deg f+ \deg h \leq q\vnum_{\frakp}(I)+(q-1) \sum_{i=1}^{m}\deg g_i$, where the last inequality follows from Lemma~\ref{lemma:alpha_q_bounded} and \ref{lemma:alphaq_equal_vnum}.
	
	If $\frakp\in \Ass I$ is generated by linear form then by above argument $\vnum(I^{[q]})\leq q\vnum_{\frakp}(I)+(q-1)\height(\frakp)$.
	Thus the second assertion follows immediately as $\vnum(I)=\min\left\{\vnum_{\frakp}(I): \frakp\in \Ass I \right\}$.
\end{proof}

\begin{proposition}\label{propviqlower}
	Let $I\subset S$ be a graded monomial ideal. Then for all $q>1$,
	$$\vnum(I^{[q]})\geq q\vnum(I)+(q-1)\height(I).$$
	The equality holds if we further assume $I$ is unmixed.
\end{proposition}
\begin{proof}
	Let $f\in S$ be a monomial and $\frakp\in\Ass(I^{[q]})$ such that $I^{[q]}:f=\frakp$ and $\deg f=\vnum(I^{[q]})$.  Since $f\frakp^{[q]}\subseteq f\frakp \subseteq I^{[q]}$, we have $f\in (I^{[q]}:\frakp^{[q]})=(I:\frakp)^{[q]}$.
	Therefore, there exists a monomial $h\in (I:\frakp)$ and a monomial $r\in S$ such that $f=h^{q}r$. Clearly, $h\not\in I$ as $f\not\in I^{[q]}$. Note that $\Ass(I^{[q]}:h^{q})=\Ass((I:h)^{[q]})=\Ass(I:h)\subseteq \Ass(I)$. Now, $((I^{[q]}:h^{q}):r)=\frakp$ implies $\frakp\in \Ass(I^{[q]}:h^q)=\Ass(I:h)$. Since $\frakp\subseteq (I:h)$ and $\frakp\in \Ass(I:h)$, we have $(I:h)=\frakp$. Therefore, $\vnum(I)\leq \deg h$.
	
	Again, observe that $\frakp=(I^{[q]}:h^qr)=(I:h)^{[q]}:r= \frakp^{[q]}:r$. Since $I$ is a monomial ideal, $\frakp$ is a prime ideal generated by linear forms. Hence, by Lemma \ref{lem:vi_ci}, we get $\vnum(\frakp^{[q]})=(q-1)\height(\frakp)\leq \deg r$. Therefore, $\vnum(I^{[q]})=q\deg h+\deg r\geq q\vnum(I)+(q-1)\height(I)$.
	
	Further, if $I$ is unmixed, so $\bight(I)=\height (I)$ and hence the final assertion follows from Proposition \ref{propviqupper}.
\end{proof}
\begin{theorem}
	\label{thm:assymptotic_of_vi}
	Let $I\subseteq S$ be a graded ideal. Then the following results hold
	\begin{enumerate}
		\item \label{thm:vi:nondec} $\vnum(I^{[q]})\geq q\vnum(I)$ for all $q\geq 1$ and hence $\left\{\frac{\vnum(I^{[q]})}{q}\right\}$ is a non-decreasing sequence in $q$, where $ q=p^e, e\in \mathbb{N}$.
		\item \label{thm:vi:lim} $\displaystyle{\lim_{q \to \infty}} \frac{\vnum(I^{[q]})}{q}$ exists.
		\item \label{thm:vi:limeq} If $I$ is unmixed and monomial, then $\displaystyle{\lim_{q \to \infty}} \frac{\vnum(I^{[q]})}{q}=\vnum(I)+\height(I)$.
	\end{enumerate}
\end{theorem}

\begin{proof}
	$(\ref{thm:vi:nondec})$: Let $f$ be a homogeneous polynomial and $\frakp \in \Ass (I^{[q]})$ such that $(I^{[q]}:f)=\frakp$ and $\deg f=\vnum(I^{[q]})$. Since $f\frakp^{[q]}\subseteq f\frakp\subseteq I^{[q]}$, we have $f\in (I^{[q]}:\frakp^{[q]})=(I:\frakp)^{[q]}$.
	Hence $f=\sum_{i=1}^{s}h_{i}^{q}r_{i}$, for some $h_{i}\in (I:\frakp)$ and $r_i\in S$ for all $1\leq i\leq s$. If $h_i\in I$ for some $i$, then we can replace $f$ by $f-h_i^qr_i$. Thus, without loss of generality, we can choose $f$ such that $f=\sum_{i=1}^{s}h_{i}^{q}r_{i}$, and $h_{i}\in (I:\frakp)\setminus I$, $r_i\in S$ for all $1\leq i\leq s$. Using the observation that $f\in (h_{1}^{q}r_1,\ldots,h_{s}^{q}r_s)$, we have
	$$\bigcap_{i=1}^{s}(I^{[q]}:h_{i}^{q}r_{i})=(I^{[q]}:(h_{1}^{q}r_1,\ldots,h_{s}^{q}r_s))\subseteq (I^{[q]}:f)=\frakp.$$
	Thus, $(I^{[q]}:h_{i}^{q}r_{i})\subseteq \frakp$ for some $i\in\{1,\ldots,s\}$. Now, observe that $(I:h_{i})^{[q]}=(I^{[q]}:h_{i}^q)\subseteq (I^{[q]}:h_{i}^{q}r_{i})\subseteq \frakp$, which imply $(I:h_{i})\subseteq \frakp$. Since $h_{i}\in (I:\frakp)\setminus I$, we have $(I:h_{i})=\frakp$. Therefore, $\vnum(I)\leq \deg h_i$ and hence $\vnum(I^{[q]})=\deg f \geq q \deg h_i \geq q\vnum(I)$.
	
	Note that $I^{[pq]}=(I^{[q]})^{[p]}$. Therefore
	$$\frac{\vnum(I^{[pq]})}{pq}\geq\frac{p\vnum(I^{[q]})}{pq}=\frac{\vnum(I^{[q]})}{q}.$$ 
	Hence, $\frac{\vnum(I^{[q]})}{q}$ is a non-decreasing sequence.
	
	$(\ref{thm:vi:lim})$: By $(\ref{thm:vi:nondec})$, we get $\frac{\vnum(I^{[q]})}{q}$ is a non-decreasing sequence of real numbers. Also, from Proposition \ref{propviqupper}, it follows that the sequence $\frac{\vnum(I^{[q]})}{q}$ is bounded above by $\vnum_{\frakp}(I)+ \sum_{i=1}^{m}\deg g_i$ for some associated prime $\frakp =(g_1,\ldots,g_m)$. Hence, $\displaystyle{\lim_{q \to \infty}} \frac{\vnum(I^{[q]})}{q}$ exists.
	
	$(\ref{thm:vi:limeq})$: If $I$ is unmixed monomial ideal, then by Proposition~\ref{propviqlower}, $\vnum(I^{[q]})= q\vnum(I)+(q-1)\height(I)$. The outcome is obtained by dividing the above expression by $q$ and taking the limit.
\end{proof}

\begin{theorem}\label{thmalphaq}
	Let $I\subseteq S$ be a graded ideal. Then, the following hold 
	\begin{enumerate}
		\item\label{thm:alph:less_vi} $\alpha_{q}(I)\leq \vnum(I^{[q]})$ for each $q>1$.
		\item \label{thm:alph:nondec} $\left\{\frac{\alpha_{q}(I)}{q} \right\}$ is a non-decreasing sequence in $q$, where $ q=p^e, e\in \mathbb{N}$.
		\item \label{thm:alph:lim}  $\displaystyle{\lim_{q\to \infty}}\frac{\alpha_{q}(I)}{q}$ exists.
		\item \label{thm:alph:geq_vi} If $I$ is a radical ideal, then there exists an associated prime $\frakp\in \Ass I$, such that $\alpha_{q}(I)\geq \vnum(I^{[q]})-\vnum_{\frakp}(I)$ for each $q>1$.
		\item \label{thm:alph:lim_equal} If $I$ is a radical ideal, then	$\displaystyle{\lim_{q\to \infty}}\frac{\alpha_{q}(I)}{q}=\displaystyle{\lim_{q\to \infty}}\frac{\vnum(I^{[q]})}{q}$.
		\item \label{thm:alph:ceil_vi} If $I$ is an unmixed square-free monomial ideal, then
		for any $q>\dim S$, we get $\lceil \frac{\alpha_{q}(I)}{q}\rceil=\vnum(I)+\height (I)$, where $\lceil z \rceil$ denotes the least integer greater than or equal to $z$, known as the ceiling function.
	\end{enumerate}
\end{theorem}

\begin{proof}
	$(\ref{thm:alph:less_vi})$: Let $f$ be a homogeneous polynomial and $\frakp \in \Ass(I^{[q]})$ such that $I^{[q]}:f=\frakp$ and $\vnum(I^{[q]})=\deg f$. %Then $f\not\in I^{[q]}$ as $(I^{[q]}:f)\neq S$. 
	Since $\Ass(I^{[q]})=\Ass(I)$, we have $fI\subseteq f\frakp \subseteq I^{[q]}$, which gives $f\in (I^{[q]}:I)\setminus I^{[q]}$. Hence, $\alpha_{q}(I)\leq \vnum(I^{[q]})$.
	
	$(\ref{thm:alph:nondec})$ It is sufficient to show $\alpha_{pq}(I)\geq p\alpha_{q}(I)$. Let $f\in (I^{[pq]}:I)\setminus I^{[pq]}$ be a homogeneous element of degree $\alpha_{pq}(I)$. Since $I^{[pq]}:I\subseteq I^{[pq]}:I^{[p]}= \left(I^{[q]}:I \right)^{[p]}$, $f=\sum_{i=1}^{s} h_i^p r_i$ for some $h_i\in I^{[q]}:I$ and $r_i\in S$, $1\leq i \leq s$. If  $h_i\in I^{[q]}$ for all $1\leq i \leq s$, then $f\in I^{[pq]}$ which is not true. Hence there exists an $i$, such that $h_i\in (I^{[q]}:I)\setminus I^{[q]}$. So $\alpha_{q}(I)\leq \deg h_i$. 
	Thus, $\alpha_{pq}(I)= \deg f = p \deg h_i +\deg r_i \geq p \alpha_q (I)$. Hence, the proof follows.
	
	$(\ref{thm:alph:lim}):$ By Lemma~\ref{lemma:alpha_q_bounded}, $\frac{\alpha_{q}(I)}{q}$ is bounded above by $\sum_{i=1}^{m}\deg g_i$, where $I=(g_1,\ldots,g_m)$. Since $\frac{\alpha_{q}(I)}{q}$ is a non-decreasing sequence  and bounded above, so $\displaystyle{\lim_{q\to \infty}}\frac{\alpha_{q}(I)}{q}$ exists.
	
	$(\ref{thm:alph:geq_vi}):$ Let $f\in (I^{[q]}:I)\setminus I^{[q]}$  be a homogeneous polynomial such that $\deg(f)=\alpha_{q}(I)$. Consider a prime ideal $\frakp \in \Ass(I^{[q]}:f)$. Note that $\Ass(I^{[q]}:f)\subseteq \Ass(I^{[q]})= \Ass(I)$. Hence there exists a homogeneous polynomial $h\in S$ such that $(I:h)=\frakp$ and $\deg(h)=\vnum_{\frakp}(I)$. Now, $\frakp=I:h\subseteq (I^{[q]}:f):h=(I^{[q]}:fh)$. Suppose $fh\in I^{[q]}$. Then $h\in (I^{[q]}:f)\subseteq \frakp=I:h$, consequently $h^2\in I$, which gives a contradiction as $I$ is radical. Therefore, $fh\not\in I^{[q]}$ and so, $I^{[q]}:fh=\frakp$ (Remark~\ref{remark:v_Iless_deg_f}). Thus, $\vnum(I^{[q]})\leq \alpha_{q}(I)+\vnum_{\frakp}(I)$. That is $\alpha_{q}(I)\geq \vnum(I^{[q]})-\vnum_{\frakp}(I)$. 
	
	$(\ref{thm:alph:lim_equal}):$ From $(\ref{thm:alph:less_vi})$ and $(\ref{thm:alph:geq_vi})$, $\vnum(I^{[q]})-\vnum_{\frakp}(I)\leq \alpha_{q}(I)\leq \vnum(I^{q})$ for some associated prime $\frakp \in \Ass I$. Therefore 
	$\displaystyle{\lim_{q\to \infty}}\frac{\alpha_{q}(I)}{q}=\displaystyle{\lim_{q\to \infty}}\frac{\vnum(I^{[q]})}{q}$.
	
	$(\ref{thm:alph:ceil_vi}):$ Since $\vnum(I^{[q]})-\vnum_{\frakp}(I)\leq \alpha_{q}(I)\leq \vnum(I^{q})$ for some associated prime $\frakp \in \Ass I$, by Proposition \ref{propviqlower}, we get
	$$(\vnum(I)+\height(I))-\frac{\vnum_{\frakp}(I)+\height(I)}{q}\leq \frac{\alpha_{q}(I)}{q} \leq (\vnum(I)+\height(I))-\frac{\height(I)}{q}.$$
	Since $\vnum_{\frakp}(I)+\height(I)\leq \dim S$ \cite[Lemma 3.4]{vedge}, so for any $q>\dim S$, $0<\frac{\height(I)}{q}\leq \frac{\vnum_{\frakp}(I)+\height(I)}{q} <1$. Next, apply the ceiling function in order to obtain the desired result.
\end{proof}

\begin{remark}\label{rmk:vnum_pol}
	{\rm
Let $I$ be an unmixed monomial ideal in a polynomial ring $S$ of any characteristic. We denote by $I^{\mathcal{P}}$ the polarization of $I$ (see \cite{faridi06} for polarization technique) and $S^{\mathcal{P}}$ denote the corresponding ring of $I^{\mathcal{P}}$. By \cite[Proposition 2.3]{faridi06}, $I^{\mathcal{P}}$ is unmixed and $\height(I^{\mathcal{P}})=\height (I)$. Again, by \cite[Corollary 3.5]{ki_vmon},
$\vnum(I^{\mathcal{P}})=\vnum(I)$. Let $p$ be a prime number greater than $\dim S^{\mathcal{P}}$. Since the $\vnum$-numbers of monomial ideals do not depend on the characteristic, we may assume $S$ is of characteristic $p$. Therefore, by Theorem \ref{thmalphaq}(\ref{thm:alph:ceil_vi}), we get 
$\vnum(I)=\vnum(I^{\mathcal{P}})=\lceil \alpha_p(I^{\mathcal{P}})\rceil-\height (I)$.
}
\end{remark}
\section{Some Questions}
\label{sec:question}
Let $I\subset S$ be a graded ideal such that all of its associated primes are generated by linear forms.

Theorem \ref{thm:level_vnum} establishes that if $S/I$ is a level algebra, then $\vnum(I)\geq \reg S/I$. We wonder whether the converse of the statement is also true.
\begin{question}
	Assume the above setup. Moreover, assume $S/I$ is Cohen-Macaulay and $\vnum(I)\geq \reg S/I$. Then is $S/I$ a level algebra?
\end{question}
The following example suggests that the Cohen-Macaulay assumption in the above question is necessary. Take $S=K[x_1,\ldots,x_4]$, and $I=(x_1x_2,x_2x_3,x_3x_4,x_1x_4)$. Then $\vnum(I)=\reg S/I=1$, and $S/I$ is unmixed, but not Cohen-Macaulay.

The example also suggests that we should inquire whether it is possible to reduce the level hypothesis while still aiming for $\vnum(I)\geq \reg S/I$. Let $\mathbf{F}_{\bullet}$ be the minimal free resolution of $S/I$ and $c=\Projdim S/I$. Due to our computational evidence, we propose the following question.
\begin{question}
	Assume the above setup. Moreover, assume $S/I$ is unmixed, and $F_c$ has a basis consisting of same-degree elements, i.e. there exists a unique integer $j$ such that the Betti number $\beta_{c,j}$ is non-zero. Then is it true that $\vnum(I)\geq \reg S/I$? Further, can we drop the unmixed assumption? 
\end{question}
The assumptions of the question are identical to those of a level algebra, with the exception that unmixedness is assumed instead of Cohen-Macaulayness. We could not come up with a counter-example to this question also.
\section{Acknowledgment}
Kamalesh Saha would like to thank the National Board for Higher Mathematics (India) for the financial support through the NBHM Postdoctoral Fellowship. Again, both authors were partially supported by an Infosys Foundation fellowship.

\printbibliography
\end{document}